\documentclass[a4paper,12pt]{amsart}
\usepackage{amssymb}
\usepackage{amsmath}
\usepackage{ifthen}
\usepackage{graphicx}
\usepackage{amsfonts}
\usepackage{amscd}
\usepackage{amsxtra}
\usepackage{color}

\newtheorem{thm}{Theorem}
\newtheorem{cor}{Corollary}
\newtheorem{lem}{Lemma}

\newtheorem{conj}{Conjecture}
\theoremstyle{definition}
\newtheorem{example}[equation]{Example}
\newtheorem{prob}[equation]{Problem}

\newcounter {own}
\def\theown {\thesection       .\arabic{own}}

\newenvironment{rem}{%
\bigskip
\noindent \textsl{{\sl Remark. }}}{\bigskip}
\newenvironment{rems}{%
\bigskip
\noindent \textsl{{\sl Remarks. }}}{\bigskip}

\newenvironment{pf}[1][]{%
 \vskip 3mm
 \noindent
 \ifthenelse{\equal{#1}{}}%
  {{\slshape Proof. }}%
  {{\slshape #1.} }%
 }%
{\qed\bigskip}

\newcounter{alphabet}
\newcounter{tmp}
\newenvironment{Thm}[1][]{\refstepcounter{alphabet}%
\bigskip%
\noindent%
{\bf Theorem \Alph{alphabet}}%
\ifthenelse{\equal{#1}{}}{}{ (#1)}%
{\bf .} \itshape}{\vskip 8pt}

\makeatletter
\newcommand{\Ref}[1]{\@ifundefined{r@#1}{}{\setcounter{tmp}{\ref{#1}}\Alph{tmp}}}
\makeatother

\newcommand{\IN}{{\mathbb N}}
\newcommand{\IC}{{\mathbb C}}
\newcommand{\ID}{{\mathbb D}}

\newcommand{\D}{{\mathbb D}}

\def\be{\begin{equation}}
\def\ee{\end{equation}}

\newcommand{\bee}{\begin{enumerate}}
\newcommand{\eee}{\end{enumerate}}

\newcommand{\blem}{\begin{lem}}
\newcommand{\elem}{\end{lem}}
\newcommand{\bthm}{\begin{thm}}
\newcommand{\ethm}{\end{thm}}
\newcommand{\bcor}{\begin{cor}}
\newcommand{\ecor}{\end{cor}}
\newcommand{\beg}{\begin{example}}
\newcommand{\eeg}{\end{example}}
\newcommand{\begs}{\begin{examples}}
\newcommand{\eegs}{\end{examples}}
\newcommand{\bdefe}{\begin{defin}}
\newcommand{\edefe}{\end{defin}}
\newcommand{\bprob}{\begin{prob}}
\newcommand{\eprob}{\end{prob}}
\newcommand{\bei}{\begin{itemize}}
\newcommand{\eei}{\end{itemize}}

\newcommand{\bcon}{\begin{conj}}
\newcommand{\econ}{\end{conj}}
\newcommand{\bcons}{\begin{conjs}}
\newcommand{\econs}{\end{conjs}}
\newcommand{\bprop}{\begin{propo}}
\newcommand{\eprop}{\end{propo}}
\newcommand{\br}{\begin{rem}}
\newcommand{\er}{\end{rem}}
\newcommand{\brs}{\begin{rems}}
\newcommand{\ers}{\end{rems}}
\newcommand{\bo}{\begin{obser}}
\newcommand{\eo}{\end{obser}}
\newcommand{\bos}{\begin{obsers}}
\newcommand{\eos}{\end{obsers}}
\newcommand{\bpf}{\begin{pf}}
\newcommand{\epf}{\end{pf}}
\newcommand{\ba}{\begin{array}}
\newcommand{\ea}{\end{array}}
\newcommand{\beq}{\begin{eqnarray}}
\newcommand{\beqq}{\begin{eqnarray*}}
\newcommand{\eeq}{\end{eqnarray}}
\newcommand{\eeqq}{\end{eqnarray*}}

\newcommand{\ds}{\displaystyle}

\newcounter{minutes}
\divide\time by 60
\newcounter{hours}
\multiply\time by 60 \addtocounter{minutes}{-\time}

\begin{document}
\bibliographystyle{amsplain}
\title[Concave univalent functions and Dirichlet finite integral]{Concave univalent functions and Dirichlet finite integral}

\thanks{
File:~\jobname .tex,
          printed: \number\day-\number\month-\number\year,
          \thehours.\ifnum\theminutes<10{0}\fi\theminutes}

\author[Y. Abu Muhanna]{Yusuf  Abu Muhanna}
\address{Y. Abu Muhanna, Department of Mathematics,
American University of Sharjah, UAE-26666.}
\email{ymuhanna@aus.edu}

\author[S. Ponnusamy]{Saminathan Ponnusamy $^\dagger $
}
\address{S. Ponnusamy,
Indian Statistical Institute (ISI), Chennai Centre, SETS (Society
for Electronic Transactions and Security), MGR Knowledge City, CIT
Campus, Taramani, Chennai 600 113, India. }
\email{samy@isichennai.res.in, samy@iitm.ac.in}


\begin{abstract}
The article deals with the class  ${\mathcal F} _{\alpha }$
consisting of non-vanishing functions $f$ that are
analytic and univalent in $\ID$ such that the complement $\IC\backslash f(\ID) $ is a convex
set, $f(1)=\infty ,$ $f(0)=1$ and the angle at $\infty $ is less than or equal to $\alpha \pi ,$ for some $\alpha \in (1,2]$.
Related to this class is the class $CO(\alpha)$ of concave univalent mappings in $\ID$,  but this differs from
${\mathcal F} _{\alpha }$ with the standard normalization $f(0)=0=f'(0)=1.$ A number of properties of these classes are discussed
which includes an easy proof of the coefficient conjecture for $CO(2)$ settled by
Avkhadiev et al. \cite{Avk-Wir-04}. Moreover, another interesting result connected with the Yamashita conjecture on Dirichlet finite integral
for $CO(\alpha)$ is also presented.
\end{abstract}

\maketitle \pagestyle{myheadings}
\markboth{Y. Abu Muhanna, and  S. Ponnusamy}{Concave univalent functions}

\section{Introduction}\label{sec1}
Let $\mathbb{D} := \{z\in \IC:\, |z| < 1 \}$  and let $\mathcal{A} $ denote the space of functions analytic in
the unit disk $\mathbb{D}$. Then  $\mathcal{A} $
is a locally convex linear topological vector space
endowed with the topology of uniform convergence over compact subsets of
$\mathbb{D}$.  By ${\mathcal B}_0$ we denote the family of functions  $\varphi \in \mathcal{A}$ with
$\varphi (0)=0$ and $\varphi (\ID)\subset \ID$. Functions in ${\mathcal B}_0$ are called Schwarz functions.
For $F\in\mathcal{A} $, let
$$s(F)= \{ f \in \mathcal{A}:\, f\prec F \},
$$
where $\prec$ denotes the usual subordination, i.e. $f=F\circ \varphi$ for some Schwarz function
$\varphi \in \mathcal{B}_0$. Thus, we have the correspondence between $s(F)$ and ${\mathcal B}_0$ given by
$s(F)= \{ F\circ \varphi: \,\varphi \in \mathcal{B}_0\}$
and note that $s(F)$ is a compact subset of $\mathcal A$.

The following result is well-known.


\begin{Thm}\cite{AbuHallen-92}\label{TheoA}
Let $f\in\mathcal{A}$ be univalent and non-vanishing in $\ID$. Suppose that the complement of $\Omega =f(\ID)$,
denoted by $\Omega ^c$, is a convex set. If $g\prec f$ then there is a probability measure $\mu $ on  $\partial \ID =\{y\in\IC:\,|y|=1\}$
so that
\[ g(z)=\int _{\partial \ID} f(yz)\,d\mu (y).
\]
\end{Thm}

In this article we are primarily interested in the class  ${\mathcal F} _{\alpha }$
consisting of non-vanishing functions $f$ that are
analytic and univalent in $\ID$ such that the complement $\IC\backslash f(\ID):=(f(\ID))^c$ is a convex
set, $f(1)=\infty ,$ $f(0)=1$ and the angle at $\infty $ is less than or equal to $\pi \alpha ,$ for some $\alpha \in (1,2].$
Related to this class is the class of concave univalent mappings in $\ID$; the family of analytic functions $f$
that map $\ID$ conformally onto a set whose complement with respect to $\mathbb{C}$ is convex and that satisfy the normalization
$f(1) = \infty$, $f(0) = f'(0) - 1 = 0 $. In addition, we impose on these functions the condition that the opening angle of $f(\ID)$ at $\infty$ is less than or equal to
$\pi\alpha$, $\alpha\in (1,2]$. As with the standard practice, we will denote the family of such functions by $CO(\alpha)$
and call it as the class of concave univalent functions \cite{Avk-Wir-05,Pom-Cruz} which has been extensively studied in the recent years.
For a detailed discussion about concave functions, we refer to \cite{Avk-Wir-06, Avk-Wir-05, BPW-09, Pom-Cruz} and the references therein.
 We note that for  $f\in CO(\alpha)$, $\alpha\in (1,2]$, the closed set $\IC\backslash f(\ID)$
is convex and unbounded. Also, we observe that $CO(2)$ contains the classes $CO(\alpha)$,
$\alpha\in (1,2]$.

In \cite{Avk-Wir-06}, it was shown that if $f\in CO(2)$ is of the form
\be\label{eqZ2}
f(z) = z + \sum_{n=2}^{\infty} a_nz^n,
\ee
then for $n\ge 2$ the inequality
\be\label{eqZ2a}
\left|a_n -\frac{n+1}{2}\right| \leq \frac{n-1}{2}
\ee
is valid and this settles the conjecture proposed in \cite{Avk-Wir-04}. In \cite{Avk-Wir-02} it was originally
conjectured that for $f\in CO(2)$ the inequality $|a_n|\geq 1$ for $n\geq 2$, is valid.

The paper is organized as follows.  In Section \ref{sec1a}, we present a couple of inequalities on hypergeometric polynomials. They are useful and have independent interest by themselves and one of them in particular implies
the coefficient inequality \eqref{eqZ2a} which is indeed a simple consequence of one of our results (Corollary \ref{cor2}).  In addition, in Section \ref{sec2}, we present a number of basic properties of the family ${\mathcal F} _{\alpha }$ and a closed form representation for functions to be in  $CO(\alpha)$ (Theorems \ref{prop1} and \ref{prop2}). A number of
new results about concave univalent functions are also presented.
Finally in Section \ref{sec3}, we present a result related to the Yamashita functional on the estimate of certain area integral
(Theorem \ref{theo4}).

\section{Inequalities related to Gaussian hypergeometric functions}\label{sec1a}
Let  $F(a,b;c;z)$ denote the usual {\em Gaussian hypergeometric} function defined by \cite{temme}
$$F(a,b;c;z)=\sum_{n=0}^\infty \frac{(a)_n(b)_n}{(c)_n(1)_n}z^n \qquad (a, b, c\in\IC, \,\, c\notin\{0,-1, -2,\ldots\}),
$$
where $(a)_n=a(a+1)\cdots (a+n-1)$ represents the {\em Pochhammer symbol} (with the convention that $(a)_0=1$).
In the case $c=-m$, $m=0,1,2,3, \ldots$, $F(a,b;c;z)$ is defined if $a=-j$ or $b=-j$, where $j=0,1,2, \ldots$ and
$j\leq m$. If ${\rm Re}\, c> {\rm Re}\, b>0$, then we have the Euler integral representation
\be\label{ireq1}
F(a,b;c;z)=\frac{\Gamma(c)}{\Gamma(b)\Gamma(c-b)}\int_{0}^{1}t^{b-1}
(1-t)^{c-b-1}(1-tz)^{-a}\, dt,
\ee
where $|\arg (1-z)|<\pi $.

 Our special emphasize here is to consider the function
\be\label{eqZ5}
\left( \frac{1+xz}{1-z}\right) ^{\alpha } =1+\sum\limits_{n=1}^{\infty }A_{n}(\alpha ,x)z^{n},
\ee
where $|x|=1$ and $x\neq -1$. Moreover, as in \cite{Todor95},  \eqref{eqZ5} may be written as
$$\left( \frac{1+xz}{1-z}\right) ^{\alpha }=\left( 1+\frac{(1+x)z}{1-z}\right) ^{\alpha }
=1+\sum_{k=1}^{\infty} {\alpha \choose k}(1+x)^k \left(\frac{z}{1-z}\right )^k
$$
and observe that for $n\geq 1$,
$$ A_{n}(\alpha ,x)= \sum_{k=1}^{n}{n-1 \choose k-1}{\alpha \choose k}(1+x)^k
$$
which may be conveniently written in the form  
\be\label{eqZ12}
A_{n}(\alpha ,x) =\alpha (1+x) F(1-n,1-\alpha ;2;1+x).
\ee
The key point at this stage is to prove that for $\alpha \in (1,2]$,
$$\left| B_{n}(\alpha ,x)\right |\leq B_{n}(\alpha ,1)
$$
where $B_{n}(\alpha ,x) =F(1-n,1-\alpha ;2;1+x)$, $n\geq 1$.

\begin{lem}\label{theo5}
For $n\in \IN$, define $B(x) :=B_{n}(\alpha ,x) =F(1-n,1-\alpha ;2;1+x).$ Then for $-1<\alpha <1$, we have $|B(x)| \leq |B(1)|.$
\end{lem}
\begin{proof}
For $0<\alpha <1$, it follows from Euler's formula \eqref{ireq1} that
$$B(x)= K\int_{0}^{1}t^{-\alpha }(1-t)^{\alpha }(1-t(1+x))^{n-1}\,dt, \quad K=\frac{1}{\Gamma (1-\alpha )\Gamma (1+\alpha )},
$$
and if $-1<\alpha <0,$ we make the change of variable $s=1-t$ and $\beta =-\alpha$ and obtain that
$$B(x)= K \int_{0}^{1}s^{-\beta }(1-s)^{\beta }(1-(1-s)(1+x))^{n-1}\,ds.
$$
Note that $1-t(1+x)$ parametrizes the chord $[1,-x]$ from $1$ to $-x,$ as $t$ varies
from $1$ to $-x$ and that $1-(1-s)(1+x)$ parametrizes the same chord, $[1,-x],$ but from $-x$ to $ 1.$ Hence treatment of the two cases
is similar. In view of this observation, without loss of generality we may assume that $0<\alpha <1.$

Now, using the change of variable $\tau := \tau (x,t) =1-t(1+x)$,  we easily see that 
$$B(x)=-K\int_{1}^{-x}\tau ^{n-1} \left( \frac{x+\tau}{1-\tau }\right)^{\alpha }\frac{d\tau}{1+x}  .
$$
Observe that $d\tau =-(1+x)dt$ and for all $x$,
$$\left\vert \int_{1}^{-x}\frac{d\tau }{1+x}\right\vert =1 ~\mbox{ and }~\left\vert \frac{d\tau }{1+x}\right\vert =dt.
$$
As $|x|=1$, it follows easily that for each $t\in [0,1]$
$$|\tau (x,t) |^2 = |1-t-tx|^2\leq (1-t)^2+t^2-2t(1-t)=(1-2t)^2=|\tau (1,t) |^2
$$
and thus, $|\tau (x,t) |\leq |\tau (1,t) |$.  Moreover, for all $x\ne -1$ and $0<t\leq 1$, we also have
$$  \frac{x+\tau (x,t)}{1-\tau (x,t)} = \frac{1-t}{t} =\frac{1+\tau (1,t)}{1-\tau (1,t)}.
$$
Using these observations, it follows that
\beqq
|B(x)|&\leq &K\int_{1}^{-x}\left |(\tau  (x,t))^{n-1}  \left( \frac{x+\tau  (x,t)}{1-\tau  (x,t)}\right)^{\alpha }\right |
\,\left |\frac{d\tau  (x,t)}{1+x}\right |\\
&\leq &K\int_{1}^{-x}\left |\tau  (1,t)\right |^{n-1} \left | \frac{1+\tau  (1,t)}{1-\tau  (1,t) }\right |^{\alpha }
\left |\frac{d\tau  (x,t)}{1+x}\right |\\
&= & K\int_{0}^{1} \left |\tau  (1,t)\right |^{n-1} \left | \frac{1+\tau  (1,t)}{1-\tau  (1,t) }\right |^{\alpha } dt\\
&= & K\int_{1}^{-1} \left |\tau  (1,t)\right |^{n-1} \left | \frac{1+\tau  (1,t)}{1-\tau  (1,t) }\right |^{\alpha }
\left |\frac{d\tau (1,t)}{2}\right | .
\eeqq
Consequently, $|B(x)|\leq |B(1)|$ and the proof is completed.
\end{proof}
%

\begin{lem}\label{lem2}
If $n\in \IN$, $\alpha >1$ and $t$ is such that $0<\alpha t <1$, then we have
$$\left |(1+x) ^{-t}A_{n}(\alpha t ,x)\right | \leq  2^{-t}A_{n}(\alpha t ,1),
$$
where $A_{n}(\alpha ,x)$ is defined by \eqref{eqZ12}.
\end{lem}
\begin{proof} From \eqref{eqZ12}, we have
$$( 1+x) ^{-t}A_{n}(\alpha t ,x)=\alpha t (1+x)^{1-t}F (1-n,1-\alpha t ,2,1+x)
$$
and, since $| (1+x)^{1-t}| \leq 2^{1-t},$ the argument in Lemma \ref{theo5} leads to the desired result.
\end{proof}

For analytic functions $f$ and $g$ in $\ID$, we denote by $f\ll g$ whenever the
Taylor coefficients of $f$ is dominated by the corresponding Taylor coefficients of $g$ in absolute sign.

\begin{lem}\label{lem3}
Let $B(x):= B_{n}(\alpha ,x)=F(1-n,1-\alpha ;2;1+x)$. Then for $n\geq 1$ and $\alpha >1$, we have
$$\left| B_{n}(\alpha ,x)\right |\leq \left|B_{n}(\alpha ,1)\right |.
$$
\end{lem}
\begin{proof}
Choose $t_{k}$'s such that  $\sum_{k=1}^{m}t_{k}=1$ and $0<\alpha t_{k} <1$ for each $k=1,\ldots, m.$ Then we can write
$$\frac{1}{1+x}\left( \frac{1+xz}{1-z}\right) ^{\alpha} =
\prod_{k=1}^{m}\frac{1}{\left( 1+x\right) ^{t_{k}}}\left( \frac{1+xz}{1-z}\right) ^{\alpha t_{k} }.
$$
By Lemma \ref{lem2},  we have
$$(1+x) ^{-t_{k}}\left( \frac{1+xz}{1-z}\right)
^{\alpha t_{k} }=( 1+x) ^{-t_{k}}\sum_{n=0}^\infty A_{n} (\alpha t_k ,x)z^n
\ll 2 ^{-t_{k}}\left( \frac{1+z}{1-z}\right)^{\alpha t_{k} }.
$$
This gives that for each $1\leq k\leq m$ and $n\in \IN$,
$$\left |(1+x) ^{-t_{k}}A_{n}(\alpha t_k,x)\right | \leq 2^{-t_{k}}\left |
A_{n}(\alpha t_k ,1)\right | .
$$
Consequently,
$$\prod_{k=1}^{m}\frac{1}{(1+x)^{t_{k}}}\left( \frac{1+xz}{1-z}\right) ^{\alpha t_{k} } \ll \prod_{k=1}^{m}\frac{1}{2^{t_{k}}}\left( \frac{1+z}{1-z}\right) ^{\alpha t_{k} }$$
which is same as
$$\frac{1}{1+x}\left( \frac{1+xz}{1-z}\right) ^{\alpha} \ll \frac{1}{2}\left( \frac{1+z}{1-z}\right) ^{\alpha }
$$
and the desired conclusion follows if we use \eqref{eqZ12}.
\end{proof}

\section{Properties of the families ${\mathcal F} _{\alpha }$ and $CO(\alpha)$}\label{sec2}

Recall that a function $f\in {\mathcal F}\subset {\mathcal A}$ is called an extreme point of ${\mathcal F}$ if $f$ cannot be written as a
proper convex combination of two distinct elements of $\mathcal F$.
Let $E{\mathcal F}$ and $H{\mathcal F}$ denote the set of extreme points of ${\mathcal F}$ and the closed convex hull of ${\mathcal F}$, respectively.
We refer to the monograph of Hallenbeck and MacGregor \cite{HM} for necessary details on this topic and also on
the class $\mathcal S$ of univalent functions $f \in{\mathcal A}$, with the standard normalization $f(0)=0$ and $f'(0)=1$.
The family  $\mathcal S$ together with some of its geometric subfamilies defined above have been intensively studied by many authors.

The following simple subordination condition will be useful in deriving coefficient conditions for functions in  ${\mathcal F} _{\alpha }$
and  $CO(\alpha)$.

\begin{thm}\label{prop1}
If $f\in {\mathcal F} _{\alpha }$ for some  $\alpha \in (1,2]$, then
$$f(z)\prec \left( \frac{1+xz}{1-z}\right) ^{\alpha } ~\mbox{ for some $|x|=1$, $x\ne -1$. }
$$
%
\end{thm}
\begin{proof}
Let $\alpha \in (1,2],$ and $f\in {\mathcal F} _{\alpha }$ with $\Omega=f(\ID)$.
For $\alpha \in (1,2)$, in  $\Omega ^c$ joins $0$ to $\infty $ by $2$ rays, each is asymptotic to one
side of the point at infinity so that the angle must be $\alpha \pi .$  The result follows.

If $\alpha =2$, then in $\Omega ^c$  joins $0$ to $\infty $ by a ray. Again, the result follows.
\end{proof}

Theorems A and  \ref{prop1}  imply that for $f\in {\mathcal F} _{\alpha }$ one has
$$f(z) =\int_{\partial \ID}\left( \frac{1+xyz}{1-yz}\right) ^{\alpha }d\mu (y) ~\mbox{ for $\alpha \in (1,2)$}
$$
and
$$f(z) =\int _{\partial \ID}\left( \frac{1+xyz}{1-yz}\right) ^{2}d\nu (y) ~\mbox{ for $\alpha = 2$},
$$
for some probability measures $\mu $ and $\nu$ on  $\partial \ID$, respectively.


\bcor
If $f\in {\mathcal F} _{\alpha }$ for some  $\alpha \in (1,2]$ and $f(z) =1+\sum _{n=1}^{\infty }a_{n}z^{n}$, then
$$|a_{n}|\leq A_{n}(\alpha )\leq A_{n}(2)= 4n~ \mbox{ for $n\ge 1$},
$$
where $A_{n}(\alpha ,1 )=:A_{n}(\alpha )$ and $A_{n}(\alpha ,x)$ is defined by \eqref{eqZ12}.
\ecor
\bpf
Suppose that $f \in {\mathcal F} _{\alpha }$, $\alpha \in (1,2]$. Then we have
from Theorem \ref{prop1} that
\beqq
|a_{n}| &=& \left |\int_{\partial \ID} A_{n}(\alpha,x )y^n\,d\mu (y)\right |\\
&\leq & \int_{\partial \ID} |A_{n}(\alpha ,x)|\,d\mu (y) =|A_{n}(\alpha ,x )|\\
&\leq &  A_{n}(\alpha )\leq A_{n}(2)=4n
\eeqq
for each $n\geq 1$. Here the second inequality is well-known (see for instance \cite{AF74}).
\epf

Next, we consider the functional  $L(f)=a_{n}-a_{1}\frac{n+1}{2}.$  We see that  $L$ is a complex linear functional such that
$|L|$ is subadditive and thus, the maximum for the functional is given by
$$\left | a_{n}-a_{1}\frac{n+1}{2}\right | \leq \left |A_{n}(\alpha
)-A_{1}(\alpha )\frac{n+1}{2}\right |\leq \left |4n-4\frac{n+1}{2}\right |=\frac{4n-4}{2},
$$
where the second inequality follows from the fact that ${\mathcal F} _{\alpha }$ is a subset of ${\mathcal F} _{2}$.
Moreover, the first inequality is attained at an extreme point of ${\mathcal F} _{\alpha }$ whereas the second inequality at extreme points of
${\mathcal F} _2$. It is worth to remark that the class ${\mathcal F} _{\alpha }$ is not the same as the class $CO(\alpha )$ of concave univalent functions
(defined above) with the standard normalization.


For convenience, denote by $CO_1(\alpha )$ the family of functions $f$ which is same as $CO(\alpha )$ but without the condition that $f(1)=\infty$.
Thus, the class  $CO_1(\alpha )$ is more general than  $CO(\alpha )$.

\begin{thm}\label{prop2}
Let $f \in CO(\alpha )$ and $b\in
\partial f(\ID)$ be a nearest point to $0.$ Then there is an $x\in\partial \ID$ with $x\ne -1$ such that
\be\label{eqZ10}
f(z)=-b\left( \frac{1+x\varphi (z)}{1-\varphi (z)}\right) ^{\alpha }+b,
\ee
where $\varphi \in{\mathcal B}_0$ is univalent in $\ID$, $\varphi (1)=1$ and $\varphi '(0)=-\frac{1}{(1+x)b\alpha }.$
Moreover, each $f\in CO_1(\alpha )$ has the form
\be\label{eqZ6}
f(z)=\frac{1}{(1+x)\alpha }\int _{\partial \ID} \frac{1}{y}\left [\left( \frac{1+xyz}{1-yz}\right) ^{\alpha }-1\right ]d\mu (y),
\ee
for some $|x|=1$, $x\ne -1$. Here $\mu $ is a probability measure on the unit circle $\partial \ID$,
$$1\geq \frac{1}{|1+x|\alpha }=|b|\geq \frac{1}{2\alpha },~\mbox{ and }~\frac{1}{\alpha }\leq |1+x|\leq 1.
$$
Furthermore, an extreme point of $CO_1(\alpha)$ has the form
$$ \frac{1}{\alpha (1+x)y}\left[ \left( \frac{1+xyz}{1-yz}\right) ^{\alpha }-1
\right], ~~x, y\in\ID,~ x\neq  -1.
$$
\end{thm}
\begin{proof}
Clearly, the function $(b-f(z))/b$ belongs to ${\mathcal F} _{\alpha }$ and thus,  Theorem \ref{prop1}
implies that
$$\frac{f(z)-b}{-b}\prec \left( \frac{1+xz}{1-z}\right)^{\alpha } ~\mbox{ for some $|x|=1$, $x\ne -1$}
$$
which gives \eqref{eqZ10}. Thus, for $f\in CO_1(\alpha )$, we must have
\be\label{eqZ11}
f(z)=-b\int _{\partial \ID} \frac{1}{y}\left [\left( \frac{1+xyz}{1-yz}\right) ^{\alpha }-1\right ]d\mu (y).
\ee
Simple computation, using \eqref{eqZ10} and \eqref{eqZ11}, shows that
$$f'(0)=-(1+x)b\alpha \varphi'(0) =-b (1+x)\alpha \int _{\partial \ID}\,d\mu (y)=-b (1+x)\alpha.
$$
Because of the normalization condition on $f$, we must have $b=-1/[(1+x)\alpha]$  and, since $d(0,\partial f(\ID))\leq 1,$ for any $f\in {\mathcal S}$,
we obtain that
\begin{equation*}
1\geq |b|=\frac{1}{|1+x|\alpha }\geq \frac{1}{2\alpha }.
\end{equation*}
The desired representation \eqref{eqZ6} follows.
\end{proof}



Theorem \ref{prop2} is handy in getting many well-known results about the family $CO(\alpha )$. We shall now state and prove some of its consequences.

Recall that the hyperbolic (or Poincar\'e) metric for $\mathbb{D}$  is the Riemannian metric
defined by $\lambda _{\mathbb{D}}(z)|dz|$, where
$\lambda _{\mathbb{D}}(z)=\frac{1}{1-|z|^{2}}$
denotes the hyperbolic density of $\ID$ with constant curvature  $-4$. Using analytic maps, hyperbolic metrics
can be transferred from one domain to another as follows. For a given hyperbolic domain $\Omega $ (i.e. its complement
contains at least two points), and a conformal map $f:\,\mathbb{D}\rightarrow \Omega $,
the hyperbolic metric of $\Omega $ is given by
$$\lambda _{\Omega }(f(z))|f'(z)|\,|dz|=\lambda _{\mathbb{D}}(z)|dz|.
$$
It is well-known that the metric $\lambda _{\Omega }$ is independent of the choice of the
conformal map $f$ used. We may now state our first consequence of Theorem \ref{prop2}.

\bcor\label{cor2a}
If $f\in CO(\alpha )$ is of the form \eqref{eqZ2} and $f(\ID)=\Omega$,
then we have
$$1\geq d(f(0),\partial \Omega)= \frac{1}{|1+x|\alpha }\geq \frac{1}{2\alpha }
$$
for some $|x|=1$, $x\ne -1$. Here $d(w,\partial \Omega )$ denotes the Euclidean distance
between $w\in \Omega $ and $\partial \Omega$, the boundary of $\Omega $.
\ecor
\br
It is also known that the image of $\ID$ under functions in $CO(\alpha )$ contains the schlicht disc
$|z|<1/(2\alpha )$, a fact which also follows from Corollary \ref{cor2a}. Moreover,  Avkhadiev and Wirths  \cite{Avk-Wir-05} have proved that
$\bigcap _{f\in CO(\alpha )}f(\D)=\{w:{\rm Re}\, w>-1/(2\alpha )\}.$ This fact is also confirmed from the subordination relation of $f\in CO(\alpha )$
and the relation on $b$ in the proof of Corollary \ref{cor2a}.
\er

\begin{cor}\label{cor1}
If $f\in CO(\alpha )$ is of the form \eqref{eqZ2} and $f(\ID)=\Omega$, then for each $a\in\ID$ we have
\[1\geq d(f(a),\partial \Omega )\lambda _{\Omega }(a)\geq \frac{1}{\alpha |x+1|}\geq \frac{1}{2\alpha }
\]
for some $|x|=1$, $x\ne -1$.
\end{cor}
\bpf
Consider a disk automorphism $\phi _a(z)=(z+a)/(1+\overline{a}z)$ and the Koebe transformation
$$F(z)= \frac{(f\circ \phi _a)(z)-f(a)}{(1-|a|^{2})f'(a)}.
$$
Then $F\in CO(\alpha )$,
$$d(0,\partial F(\ID))=\underset{t\in [0,2\pi]}{\inf }\left | \frac{f(\phi _a(e^{it}))-f(a)}{(1-|a|^{2})f' (a)}\right |
$$
and, by Corollary \ref{cor2a}, we obtain
$$1\geq d(0,\partial F(\ID))\geq \frac{1}{|1+x|\alpha }\geq \frac{1}{2\alpha }.
$$
Hence, we find that
$$1\geq d(f(a),\partial f(\ID))\lambda _{f(\ID)}(a)\geq \frac{1}{2\alpha }
$$
and the desired inequality follows.
\epf

In \cite[Corollary 2.2]{Bhow-12}, it was shown that if $f\in CO(\alpha )$ is of the form \eqref{eqZ2}, then $|a_{n}|\leq \frac{1}{2\alpha} A_{n}(\alpha )$ for $n\geq 2$. In our next result, we show this as a simple consequence of Theorem \ref{prop2}.

\begin{cor}\label{cor2}
If $f\in CO(\alpha )$ is of the form \eqref{eqZ2},  $f(\ID)=\Omega$, and $\alpha \in (1,2]$,
then $|a_{n}|\leq  \frac{1}{2\alpha}A_{n}(\alpha )$ holds for $n\geq 2$, where each coefficient
$A_{n}(\alpha ,1 )=A_{n}(\alpha )$ is given by \eqref{eqZ5}. In particular,  for $f\in CO(2)$, we have the sharp estimate
\be\label{eqZ9}
\left |a_{n}-\frac{n+1}{2}\right |\leq \frac{n-1}{2}
\ee
for $n\geq 2.$
\end{cor}
\begin{proof}
Let $f(z) = z + \sum_{n=2}^{\infty} a_nz^n$. Then the equation \eqref{eqZ6} obviously implies that
\be\label{eqZ13}
a_{n} = B_n(\alpha,x ) \int_{\partial \ID} y^{n-1}\,d\mu (y), \quad B_n(\alpha,x ) =\frac{A_n(\alpha,x )}{\alpha (x+1)} ~\mbox{ for $n\geq 1$},
\ee
where $A_n(\alpha,x )$ is defined by \eqref{eqZ5}. By \eqref{eqZ13} and Lemma \ref{lem3}, it follows that
$$|a_{n}|\leq |B_n(\alpha,x ) |\leq B_n(\alpha,1 )   ~\mbox{ for $|x|=1$, $x\neq -1$}
$$
and the desired inequality follows.

From Theorem \ref{prop2}, the extremal function for the coefficients of $CO(\alpha)$ is
$$f(z)=\frac{1}{2\alpha}\left [\left( \frac{1+z}{1-z}\right) ^{\alpha}-1\right ]=z+ \frac{1}{2\alpha}\sum_{n=2}^{\infty}A_{n}(\alpha )z^n
$$
showing that the coefficient inequality $|a_{n}|\leq \frac{1}{2\alpha}A_{n}(\alpha )$ is sharp.
%

Next, we note that $L(f)=a_{n}-a_{1}\frac{(n+1)}{2}$ is a complex linear functional such that $|L|$ is subadditive and
$CO(\alpha )\subset CO(2).$ (Note that the maximum of $|L|$ is attained at extremal function).
From \eqref{eqZ6}, extremal functions for $CO(2)$ are given by
$$f_0(z)=-b\left( \frac{1+xyz}{1-yz}\right) ^{2}+b=\sum_{k=1}^{\infty}a_{k}(2)z^{k},
$$
where $b=-\frac{1}{2(1+x)}$ is the nearest point on $\partial f_0(\ID)$ to $0.$ But $f_0$ is a normalized slit map and hence,
we must have $|b|=1/4$ which shows that the maximal for the functional is given by
$$\left |a_{n}-\frac{n+1}{2}\right | \leq \left |a_{n}(2)-
\frac{n+1}{2} \right |\leq \left |n-\frac{n+1}{2}\right|=\frac{n-1}{2}
$$
which holds for $n\geq 2$ and the extremal function is given by the univalent functions
$$f_{\theta}(z)=\frac{1}{2(1+e^{i\theta})}\left(\left(\frac{1+e^{i\theta } z}{1-z}\right)^2 -1\right)=
\frac{z-\frac{1}{2}(1-e^{i\theta})z^2}{(1-z)^2} \mbox{ for $\theta \in [0, 2\pi]\setminus \{\pi\}$}
$$
which map the unit disk $\ID$ onto the complex plane minus a (possibly skew) half-line, whereas
$f_{\pi}(z)=\frac{z}{1-z}$
maps the disk $D$ onto a half plane. It is easy to check that equality holds in \eqref{eqZ9} for every $n$
if $f=f_{\theta}$, $\theta \in [0, 2\pi]$.
\end{proof}

\section{Yamashita's area integral}\label{sec3}

We denote the (Euclidean) area of the image of the disk $\ID_r$ under an analytic function $g$ by $\Delta(r,g)$, where $0<r\leq1$.
Thus
$$\Delta(r,g)=\int\int_{\ID_r}|g'(z)|^{2}\,d\sigma ,
$$
where $d\sigma$ denotes the area element.
We call $g$ a Dirichlet-finite function if $\Delta(1,g)$, the area covered by the mapping $z\rightarrow g(z)$ for $|z|<1$,
is finite and in this case, we say that $g$ has finite Dirichlet integral.
Thus, a function has finite Dirichlet integral exactly when its image has finite area (counting multiplicities).
In 1990,   Yamashita \cite{Yama-1990} proved the following.

\begin{Thm}
We have for the Yamashita functional,
\bee
\item[{\rm (1)}]
$ \displaystyle \max_{f\in\mathcal S}\Delta\left( r,\frac{f(z)}{z}\right)=\frac{2\pi r^{2}(r^{2}+2)}{(1-r^{2})^{4}}~\mbox{ for $0<r\leq 1$.}$
\item[{\rm (2)}]
$\displaystyle \max_{f\in\mathcal S}\Delta \left(r,\frac{z}{f(z)}\right)=2\pi r^{2}(r^{2}+2) ~\mbox{ for $0<r\leq 1$.}
$
\eee
For each $r$, $0<r\leq1$, the maximum is attained only by the rotations of the Koebe function $k(z)$.
In particular, $z$ and  $z/f(z)$ for $f\in{\mathcal S}$ are bounded and Dirichlet-finite in $\ID$ with
$\Delta (1,z/f(z))\leq 6\pi$ and $|z/f(z)|\leq 4$ in $\ID$.
\end{Thm}

For the family ${\mathcal C}\subset {\mathcal S}$ of convex functions,
Yamashita \cite[p.~439]{Yama-1990} conjectured that
$$\max_{f\in\mathcal C}\Delta \left(r,\frac{z}{f(z)}\right)=\pi r^{2}, ~\mbox{ for $0<r\leq 1$},
$$
where the maximum is attained only by the rotations of the function $j(z)=z/(1-z)$.
In \cite{ObPoWi2013}, the authors proved this conjecture in a more general form and in \cite{ObPoWi2015,PoWi2014}, the
authors obtained analog result for some other classes of functions and spirallike functions, respectively.
In this section, we obtain a similar result (Theorem \ref{theo4}) for concave functions but with the generalized Yamashita functional (i.e.
instead $z/f(z)$ we consider $\varphi (z)/f(z)$ in the area integral) which,  in view of the representation of functions
in $Co(\alpha)$ given by Theorem \ref{prop2}, is more natural.

\begin{lem}\label{lem4}
Suppose that $f(z)=F(\varphi (z))\in CO(\alpha )$ for some univalent function $\varphi \in {\mathcal B}_0$, where
$$F(z)=-b\left( \frac{ 1+xz}{1-z}\right) ^{\alpha }+b
$$
and $b\in \partial f(\ID)$ is the nearest point to $0$. Then we have
\[
\int\int_{\ID}\left |\left( \frac{\varphi (z)}{f(z)}\right) '\right |^{2}\,dxdy= \int\int_{\ID}\left |\left( \frac{w}{F(w)}\right) '\right |^{2}\,dA_{w}.
\]
\end{lem}
\begin{proof}
Clearly
$$\left( \frac{\varphi}{f}\right)'(z)=\left( \frac{\varphi }{
F\circ \varphi }\right)' (z) =\frac{F(\varphi (z))-\varphi (z)F'(\varphi (z))}{\left( F(\varphi (z))\right) ^{2}}\varphi '(z).
$$
We may now, let $dA_{w}=|\varphi '(z)|^{2}\,dxdy$ so that, since $\varphi $ is univalent,
\[
\int\int_{\ID}\left |\left( \frac{\varphi}{f}\right)'(z)\right |^{2}\,dxdy= \int\int_{\varphi (\ID)}\left |\left( \frac{I}{F}\right)'(w)\right |^{2}\,dA_{w}
\]
where $I(w)=w$. The proof is complete.
\end{proof}



We shall need the following well-known lemma which may be thought of as Green's formula for analytic functions.
But for the sake of completeness, we present the proof.

\begin{lem}\label{lemC}
If $g$ is a non-vanishing analytic function in $\ID$ then
$$A_r =\int\int_{\ID_{r}} |g^{\prime }(z)|^{2}\,d\sigma=\frac{1}{2}\int_{0}^{2\pi
}|g(re^{i\theta })|^{2}{\rm Re}\left( \frac{-re^{i\theta }g^{\prime }(re^{i\theta })}{g(re^{i\theta })}\right) d\theta ,
$$
where $d\sigma$ denotes the area element $dx\,dy$
\end{lem}
\begin{proof}
%
By Green's theorem, the area $A_r$ of $\Omega _r =g(\ID_r)$ is given by the contour integral
$$A_r = -\frac{1}{2i} \int _{\partial \Omega _r}\overline{w}\,dw
$$
which after substituting $w=g(z)$ with $z=re^{i\theta}$ gives
$$A_r= -\frac{1}{2i} \int _{ |z|=r} \overline{g(z)}g'(z)\,dz=\frac{1}{2}\int_{0}^{2\pi }|g(z)|^{2}{\rm Re}\left( \frac{-zg^{\prime }(z)}{g(z)}\right)d\theta
$$
and the desired conclusion follows.
\end{proof}

\begin{thm}\label{theo4}
Assume the hypotheses of Lemma \ref{lem4}.
Then we have
$$\int\int_{\ID_r}\left |\left( \frac{\varphi }{f}\right)' (z)\right |^{2}\,dxdy\leq M\pi r^{2},
$$
where
$$M=\frac{1}{\alpha |b|^{2}}\left[ \frac{\alpha -1}{4}+\frac{1}{4}+ E_0
\right]
~\mbox{ and }~ E_0=\max_{|\gamma |\leq \gamma_0} E(\gamma).
$$ 
Here
\be\label{eqZ30a}
E(\gamma ) = \left \{\begin{array}{rl}
\ds \frac{\tan (\gamma /2)}{4\tan ((\pi -\gamma)\alpha/2)} & \mbox{ if $0\leq \gamma <\pi$}\\[4mm]
\ds \frac{-\tan (\gamma /2)}{4\tan ((\pi +\gamma)\alpha/2)} & \mbox{ if $-\pi<\gamma <0$} \end{array}\right .
\ee
and
\be\label{eqZ30}
\gamma _0:= \gamma _0(\alpha, |b|)=
\left \{\begin{array}{cl}
\ds \pi - \arctan \left (\frac{\sqrt{4\alpha ^2|b|^2-1}}{|2\alpha ^2|b|^2 -1|}\right )& ~\mbox{ if $2\alpha ^2|b|^2 \geq 1$}\\[2mm]
\ds \arctan \left (\frac{\sqrt{4\alpha ^2|b|^2-1}}{|2\alpha ^2|b|^2 -1|}\right ) & ~\mbox{ if $1/2\leq 2\alpha ^2|b|^2 <1$}.\end{array}\right.
\ee
\end{thm}
\begin{proof}
Given $f(z)=F(\varphi (z))\in CO(\alpha )$. We see that $F(0)=0$, $F'(0)=- b\alpha (1+x)$ and $F''(0)=-b\alpha (1+\alpha) (1+x)$
which gives
$$F(z)=-b\alpha (1+x)z -\frac{b\alpha (1+x)[(\alpha -1)(1+x)+2]}{2!}z^2 +\cdots
$$
and
$$\varphi '(0)=\frac{f'(0)}{F'(0)}=-\frac{1}{b\alpha (1+x)}.
$$
We now set $g(z)=z/{F(z)}.$ Then, $g(z)$ is non-vanishing in $\ID$ and we observe that
\be\label{eqZ3}
g(z)=-\frac{1}{b\alpha (1+x)}+\cdots
\ee
and
\be\label{eqZ4}
zg'(z)=\left (1-\frac{zF'(z)}{F(z)}\right )g(z).
\ee
Since
\beqq
\frac{zF'(z)}{F(z)} &= &\frac{-b\alpha z\left ( \left( \frac{1+xz}{1-z}\right) ^{\alpha -1}\cdot \frac{1+x}{( 1-z)^2}\right)}{-b\left(
\frac{1+xz}{1-z}\right) ^{\alpha }+b}\\
&=& \frac{-\alpha (1+x)z-\alpha (1+x)[(\alpha -1)(1+x)+2]z^2 +\cdots }{-\alpha (1+x)z-\alpha (1+x)[(\alpha -1)(1+x)+2](1/2)z^2 +\cdots}\\
&=&\frac{1+ [(\alpha -1)(1+x)+2]z +\cdots }{1+ [(\alpha -1)(1+x)+2](1/2)z +\cdots}\\
&=&1+\frac{1}{2}\left[ (\alpha -1)(1+x)+2%
\right] z+\cdots,
\eeqq
we find that
\beqq
1-\frac{zF'(z)}{F(z)}&=&\frac{-\left( \frac{1+xz}{1-z}\right)
^{\alpha }+1+\alpha z\left ( \left( \frac{1+xz}{1-z}\right) ^{\alpha -1}\cdot \frac{1+x}{( 1-z)^{2}}\right ) }{-\left( \frac{1+xz}{1-z}\right)
^{\alpha }+1}\\
&=&-\frac{1}{2}\left[ (\alpha -1)(1+x)+2\right] z+\cdots
\eeqq
and thus, using \eqref{eqZ3} and \eqref{eqZ4}, we have
\beqq
zg'(z)\overline{g(z)} & = &\frac{-\left( \frac{1+xz}{1-z}\right)
^{\alpha }+1+\alpha z\left( \left( \frac{1+xz}{1-z}\right) ^{\alpha -1}\cdot \frac{1+x}{(1-z)^2}\right )}{-\left( \frac{1+xz}{1-z}\right)
^{\alpha }+1}.\frac{z}{-b\left( \frac{1+xz}{1-z}\right) ^{\alpha }+b} \left (
\frac{\overline{z}}{-\overline{b}\left( \frac{1+\overline{x}\overline{z}}{1-\overline{z}}\right) ^{\alpha }+\overline{b}}\right )\\
&=&\left( -\frac{1}{2}\left[ (\alpha -1)(1+x)+2\right] z+ \cdots \right) \left(
\frac{1}{-b\alpha (1+x)}+ \cdots \right)  \frac{\overline{z}}{-\overline{b}\left(
\frac{1+\overline{x}\overline{z}}{1-\overline{z}}\right) ^{\alpha }+\overline{b}}\\
&=& \left( -\frac{\left[ (\alpha -1)(1+x)+2\right] }{-2b\alpha (1+x)}z+ \cdots \right) \frac{\overline{z}}{-\overline{b}\left( \frac{1+\overline{x}
\overline{z}}{1-\overline{z}}\right) ^{\alpha }+\overline{b}}.
\eeqq
Now, we fix $r$ and replace $z=re^{i\theta }.$ Then the last relation  becomes
%
\beqq
\frac{zg'(z)}{g(z)}|g(z)|^2&=&\frac{1}{e^{i\theta}}\left[ \frac{ -\left[ (\alpha
-1)(1+x)+2\right] }{-2b\alpha (1+x)}re^{i\theta}+\cdots \right] \cdot \left[ \frac{r}{
-\overline{b}\left( \frac{e^{i\theta }+\overline{x}r}{e^{i\theta }-r}\right)
^{\alpha }+\overline{b}}\right ]\\
&=:& e^{-i\theta }A_r(e^{i\theta }) B_r(e^{i\theta }),
\eeqq
where the functions $A_r(e^{i\theta })$ is clearly analytic and bounded in $\ID$ and in the
variable $\zeta = e^{i\theta }$. Also the function
$$B_r(e^{i\theta })=\frac{r}{-\overline{b}\left( \frac{e^{i\theta }+\overline{x}r}{e^{i\theta
}-r}\right) ^{\alpha }+\overline{b}}
$$
is analytic and bounded in the variable $\zeta =e^{i\theta }$, except at the singularities at $-\overline{x}r$ and $r$.
 Hence, by Lemma \ref{lemC} and the above representation, we rewrite the area integral as
\begin{eqnarray*}
\int\int_{\ID_{r}}|g'(z)|^{2}\,d\sigma &=&\frac{1}{2}\int\limits_{0}^{2%
\pi }|g(re^{i\theta})|^{2}{\rm Re}\left( \frac{-re^{i\theta}g'(re^{i\theta})}{g(re^{i\theta})}\right) d\theta \\
&=&\frac{1}{2}{\rm Re}\, \int\limits_{0}^{2\pi }\frac{A_r(e^{i\theta })}{
e^{i\theta }}. B_r(e^{i\theta }) \frac{d(e^{i\theta })}{ie^{i\theta }} \\
&=&\frac{1}{2}{\rm Re}\,\left (\frac{1}{i}\int_{|\zeta|=1 }G(\zeta )\,d\zeta\right ),
\end{eqnarray*}
where
$$G(\zeta )=\frac{A_r(\zeta)}{\zeta}\cdot\frac{r}{\zeta \left[ -\overline{b}\left( \frac{
\zeta +\overline{x}r}{\zeta -r}\right) ^{\alpha }+\overline{b} \right]}.
$$
As $G(\zeta )$ is analytic in the annulus $r<|\zeta|\leq 1,$ it follows that
$$\frac{1}{2}{\rm Re}\,\left (\frac{1}{i}\int_{|\zeta|=1 }G(\zeta )\,d\zeta\right ) =
\frac{1}{2}{\rm Re}\,\left (\frac{1}{i}\int_{|\zeta|=\rho }G(\zeta )\,d\zeta\right ), \quad \mbox{for $r<\rho \leq 1$.}
$$
Draw a small tube around $r$ to the circle $|\zeta |=\rho $ and another
disjoint one around $-\overline{x}r$ to $-\overline{x}\rho$.  Let $\gamma _{\rho }$ be the counterclockwise
closed path consisting of parts of $|\zeta |=\rho $ and the two tubes. Then by the Cauchy residue theorem, we have
By the Cauchy residue theorem (see for example, \cite{PS-CVbook}), we have
\beqq
\frac{1}{2\pi }\int_{\gamma _{\rho }}G(\zeta )\,d\zeta &=&  {\rm Res\,}[G(\zeta );0]\\
& = &\frac{\left[ (\alpha -1)(1+x)+2\right]r }{2b\alpha (1+x)}\cdot \frac{r}{-\overline{b}(-\overline{x})^{\alpha }+\overline{b}}\\
&=&\frac{r^{2}}{2\alpha |b|^{2}}\left[ \frac{\alpha -1}{1-(-\overline{x})^{\alpha }}+
\frac{2}{(1+x)(1-(-\overline{x})^{\alpha })}\right].
\eeqq
Now, we let $\rho \searrow r.$ Then $\gamma _{\rho }$ approaches the circle $|\zeta |=r$ and we conclude that
$$\frac{1}{2\pi }\int_{\gamma _{\rho }}G(\zeta )\,d\zeta \rightarrow
\frac{1}{2\pi }\int_{|\zeta |=r}G(\zeta )\,d\zeta =\int_{|\zeta|=1 }G(\zeta )\,d\zeta.
$$
Using these observations, we thus obtain that
\begin{eqnarray}
\int\int_{\ID_{r}}|g'(z)|^{2}\,d\sigma
&=&\frac{\pi r^{2}}{2\alpha |b|^{2}} {\rm Re} \left[ \frac{\alpha -1}{1-(-\overline{x})^{\alpha }}+
\frac{2}{(1+x)(1-(-\overline{x})^{\alpha })}\right] \nonumber\\
&=&\frac{\pi r^{2}}{\alpha |b|^{2}}\left[ \frac{\alpha -1}{4}+{\rm Re}\, D(x) \right],
\label{eqZ13a}
\end{eqnarray}
where
\be\label{eqZ8}
D(x)=\frac{1}{(1+x)(1-(-\overline{x})^{\alpha })}
\ee
and, because $|x|=1$ and $(-\overline{x})^{\alpha }=e^{\alpha \log (-\overline{x})}
=e^{i\alpha {\rm Arg}\, (-\overline{x})}\in \partial \ID$,
$${\rm Re} \left (\frac{1}{1-(-\overline{x})^{\alpha }}\right )=\frac{1}{2}.
$$
Here ${\rm Arg}\, z$ represent the principle argument of the complex number $z$ so that  $-\pi <{\rm Arg}\, z\leq \pi$.
To complete the proof, we need to obtain an upper bound for ${\rm Re}\, D(x)$. We begin to
note that $|\varphi '(0)|\leq 1$ by the Schwarz lemma and, by Theorem \ref{prop2}, we also have relation
$$\alpha b(1+x)\varphi '(0)=-1.
$$
These observations imply that $1\geq |b|\geq \frac{1}{2\alpha }$ and $\frac{1}{\alpha }\leq \frac{1}{\alpha |b|}\leq 2.$ On the other hand, as $x\in \partial\ID \backslash \{-1\}$, we obtain
$$2\geq |x+1|=\left |\frac{-1}{\alpha b\varphi '(0)}\right |\geq \frac{1}{\alpha |b|}\geq \frac{1}{\alpha }\geq \frac{1}{2}.
$$
and, in particular, this inequality  with $x=e^{i\gamma }$ shows that
$$1\geq \frac{1}{2}|1+e^{i\gamma }|=|\cos (\gamma /2)|\geq \frac{1}{2\alpha |b|} ~\mbox{ and }~ {\rm Re}\, x\geq -c, \quad  c=1-\frac{1}{2\alpha ^2|b|^2}.
$$
These observations imply that
$$|\gamma |=|{\rm Arg}\,(x)|\leq  
\gamma _0 =\left \{\begin{array}{cl}
\ds \pi - \arctan \left (\frac{\sqrt{1-c^2}}{|c|}\right )& \mbox{ if $c\geq 0$}\\[1mm]
\ds \arctan \left (\frac{\sqrt{1-c^2}}{|c|} \right ) & \mbox{ if $c<0$},\end{array}\right.
$$
where $\gamma _0:= \gamma _0(\alpha, |b|)$ is obtained from substituting the above values of $c$ and thus, is given explicitly by \eqref{eqZ30}.
Next, by setting $x=e^{i\gamma }$, we find that
$$(-\overline{x})^{\alpha }=e^{i\alpha {\rm Arg}\, (-\overline{x})}  = \left \{\begin{array}{rl}
e^{i\alpha (\pi -\gamma )}& \mbox{ if $0\leq \gamma <\pi$}\\[1mm]
e^{-i\alpha (\pi +\gamma)} & \mbox{ if $-\pi<\gamma <0$.}\end{array}\right.
$$
Moreover, $1+x=2e^{i\gamma /2} \cos (\gamma /2)$, and
$$1-(-\overline{x})^{\alpha } =
\left \{\begin{array}{rl}
-2ie^{i\alpha (\pi-\gamma)/2} \sin (\alpha (\pi-\gamma) /2) & \mbox{ if $0\leq \gamma <\pi$}\\[1mm]
2ie^{-i\alpha (\pi+\gamma)/2} \sin (\alpha (\pi+\gamma) /2) & \mbox{ if $-\pi<\gamma <0$.}\end{array}\right .
$$
Using these in the expression for $D(x)$ given by \eqref{eqZ8}, we find that
$$D(x) = \left \{\begin{array}{rl}
\ds \frac{ie^{-i(\gamma +\alpha (\pi-\gamma ))/2 }}{4 \cos (\gamma /2) \sin ((\pi-\gamma )\alpha /2)} & \mbox{ if $0\leq \gamma <\pi$}\\[4mm]
\ds \frac{-ie^{-i(\gamma +\alpha (\pi +\gamma ))/2 }}{4 \cos (\gamma /2) \sin ((\pi+\gamma )\alpha /2)} & \mbox{ if $-\pi<\gamma <0$.}\end{array}\right .
$$
Using the addition formula of sine function, it is a simple exercise to see that
$${\rm Re}\,D(x) =: \frac{1}{4}+ E(\gamma ),
$$
where $E(\gamma )$ is defined by \eqref{eqZ30a}. By the definition of $E_0$, we have $E(\gamma )\leq E_0$
for $|\gamma |\leq \gamma _0$.
Finally, using the relation \eqref{eqZ13a} and the above estimate, we have
$$\int\int_{\ID_{r}}|g'(z)|^{2}\,d\sigma \leq \frac{\pi r^{2}}{\alpha
|b|^{2}}\left[ \frac{\alpha -1}{4}+\frac{1}{4} +E_0\right]
$$
and the proof is complete.
\end{proof}

\br 
From \eqref{eqZ30a}, a routine computation gives that
$$E'(\gamma )=  \left \{\begin{array}{rl}
\ds \frac{ \sin ((\pi-\gamma)\alpha ) +\alpha \sin (\gamma)}{16\cos ^2(\gamma /2)\sin ^2( (\pi-\gamma)\alpha/2)}
& \mbox{ if $0\leq \gamma <\pi$}\\[4mm]
\ds - \frac{ \sin ((\pi+\gamma)\alpha ) -\alpha \sin (\gamma)}{16\cos ^2(\gamma /2)\sin ^2( (\pi +\gamma)\alpha/2)}
& \mbox{ if $-\pi<\gamma <0$}, \end{array}\right .
$$
where $\alpha \in (1,2]$. We need to find an upper bound for $E(\gamma )$ when $|\gamma| \leq \gamma _0$, and because of the symmetry with respect to the range values of $\gamma$,
it suffices to assume that $0\leq \gamma \leq  \gamma _0$, where, with the help of the value of $c$ above, $\gamma _0$ is obtained from
\eqref{eqZ30}.
After a careful computation, it follows that for $\gamma \in (0,\gamma _0(\alpha, |b|)]$, 
$$E(\gamma ) \leq  E_0= E(\gamma _0 )=  \frac{\tan (\gamma _0 /2)}{4\tan ((\pi -\gamma _0)\alpha  /2)}.
$$
\er


\subsection*{Acknowledgments}

The authors thank the referee and Prof. K.-J. Wirths for their suggestions and comments which improve the manuscript considerably.
The second author is on leave from IIT Madras. The research of the second author was supported by the project RUS/RFBR/P-163 under
Department of Science \& Technology (India).

\end{document}